\documentclass[12pt,reqno]{amsart}

\usepackage{tikz-cd}
\usepackage[all]{xy}
\usepackage{amssymb,hyperref}

\numberwithin{equation}{section}
 
\newtheorem{theorem}{Theorem}[section]
\newtheorem{proposition}[theorem]{Proposition}
\newtheorem{lemma}[theorem]{Lemma}

\theoremstyle{definition}
\newtheorem{definition}[theorem]{Definition}
\newtheorem{remark}[theorem]{Remark}
\newtheorem{example}[theorem]{Example}

\begin{document}

\baselineskip=15pt

\title[Logarithmic Cartan geometry on complex manifolds]{Logarithmic Cartan geometry on complex
manifolds with trivial logarithmic tangent bundle}

\author[I. Biswas]{Indranil Biswas}

\address{Mathematics Department, Shiv Nadar University, NH91, Tehsil
Dadri, Greater Noida, Uttar Pradesh 201314, India}

\email{indranil29@gmail.com, indranil@math.tifr.res.in}

\author[S. Dumitrescu]{Sorin Dumitrescu}

\address{Universit\'e C\^ote d'Azur, CNRS, LJAD, France}

\email{dumitres@unice.fr}

\author[A. S. Morye]{Archana S. Morye}

\address{School of Mathematics and Statistics, University of Hyderabad, Gachibowli, Central University P O, Hyderabad
500046, India}

\email{sarchana.morye@gmail.com}



\date{}

\begin{abstract}
Let $M$ be a compact complex manifold, and $D\, \subset\, M$ a reduced normal crossing divisor on it, 
such that the logarithmic tangent bundle $TM(-\log D)$ is holomorphically trivial. Let ${\mathbb A}$ 
denote the maximal connected subgroup of the group of all holomorphic automorphisms of $M$ that preserve 
the divisor $D$. Take a holomorphic Cartan geometry $(E_H,\,\Theta)$ of type $(G,\, H)$ on $M$, where 
$H\, \subset\, G$ are complex Lie groups. We prove that $(E_H,\,\Theta)$ is isomorphic to $(\rho^* 
E_H,\,\rho^* \Theta)$ for every $\rho\, \in\, \mathbb A$ if and only if the principal $H$--bundle $E_H$ 
admits a logarithmic connection $\Delta$ singular on $D$ such that $\Theta$ is preserved by the 
connection $\Delta$.
\end{abstract}

\maketitle

\tableofcontents

\section{Introduction}\label{se1}

Let $G$ be a connected complex Lie group with Lie algebra $\mathfrak g$ and $H\, \subset\, G$ a 
closed connected complex Lie subgroup. A holomorphic Cartan geometry of type $(G,\,H)$ on 
a connected complex manifold $M$ is a pair of the form $(E_H,\,\Theta)$, where $E_H$ is a holomorphic 
principal $H$-bundle over $M$, and
$$
\Theta\,:\,T{E_H}\,\,\xrightarrow{\,\,\,\,\sim\,\,\,}\,\, E_H \times{\mathfrak g}
$$
is a holomorphic isomorphism of vector bundles, such that
\begin{enumerate}
\item $\Theta$ is $H$--equivariant, and

\item the restriction of $\Theta$ to every fiber of $E_H$ coincides with the Maurer-Cartan form
on the fiber for the action of $H$ on it.
\end{enumerate}
(See \cite{Sh}, \cite{BD}.) Note that the first condition implies that $\dim M\,=\, \dim G/H$. Many geometric structures are
special cases of Cartan geometries.

Here we consider $M$ to be a connected compact complex manifold equipped with a normal crossing divisor 
$D\, \subset\, M$ such that the logarithmic tangent bundle $TM(-\log D)$ is holomorphically trivial. 
Such pairs $(M,\, D)$ were classified in \cite{Wi} (recalled here in Theorem \ref{th1} and 
Theorem \ref{th2}). Consider the group of all holomorphic automorphisms of $M$ that preserve $D$. Let 
$\mathbb A$ denote that maximal connected subgroup of it. This $\mathbb A$ is a connected complex Lie 
group that acts transitively on the complement $M\setminus D$. Take a holomorphic Cartan geometry 
$(E_H,\,\Theta)$ type $(G,\, H)$ on $M$. It is called weakly homogeneous if $(E_H,\,\Theta)$ is 
isomorphic to the holomorphic Cartan geometry $(\rho^* E_H,\,\rho^* \Theta)$ for every $\rho\, \in\, 
\mathbb A$.

Our main theorem is the following (see Theorem \ref{thm2}):

\textit{A holomorphic Cartan geometry $(E_H,\,\Theta)$ of type $(G,\, H)$ on $M$ is weakly homogeneous if and only if
the principal $H$--bundle $E_H$ admits a logarithmic connection $\Delta$ singular on $D$ such
that $\Theta$ is preserved by the connection $\Delta$.}

See Definition \ref{def5} for the above stated condition that $\Delta$ preserves $\Theta$.

Section \ref{section Log Cartan} extends Theorem \ref{thm2} to the broader class of logarithmic 
Cartan geometries, which were introduced in \cite{BDM}. As 
before, a logarithmic Cartan geometry is called weakly homogeneous if its isomorphism class 
does not change under the pullback operation through the elements of the automorphism group 
$\mathbb A$. Take any logarithmic Cartan geometry on $(M,\, D)$ of type $(G,\, H)$. If it is 
weakly homogeneous, then it is shown that the underlying holomorphic principal $H$--bundle on 
the complement $M\setminus D$ admits a special logarithmic connection (see Theorem \ref{thm1}).

\section{Cartan geometry and group action}

\subsection{Holomorphic Cartan geometry}

We shall denote by $G$ a connected complex Lie group; its Lie algebra will
be denoted by $\mathfrak g$. Let $H\,<\,G$ be a closed connected complex Lie subgroup with
Lie algebra ${\mathfrak h}\, \subset\, {\mathfrak g}$. Let $M$ be a connected compact complex manifold
and
\begin{equation}\label{x1}
\pi\,:\,E_H\,\longrightarrow\, M
\end{equation}
a holomorphic principal $H$-bundle over $M$. The holomorphic tangent bundle of $E_H$ will be denoted
by $T{E_H}$. For any $h\, \in\, H$, let
$$
R_h\, :\, E_H\,\longrightarrow\, E_H,\, \ z\, \longmapsto\, zh
$$
be the automorphism given by $h$. We have the differential
$$
dR_h\, :\, TE_H\,\longrightarrow\, TE_H
$$
of the above map $R_h$.
The action of the group $H$ on $E_H$ produces an action of $H$ on the tangent bundle $T{E_H}$. In other
words, for $p \,\in\, E_H$, a tangent vector $v\,\in\, T_p E_H$ and $h\,\in\, H$,
$$v\cdot h \,=\, dR_h (v).$$

Let
\begin{equation}\label{x-1}
d\pi \,:\, T{E_H} \,\longrightarrow\, \pi^* TM
\end{equation}
be the differential of the projection $\pi$ in \eqref{x1}. The action of $H$ on $E_H$ identifies the
kernel ${\ker}(d\pi)\, \subset\, T{E_H}$ of $d\pi$ with the trivial vector bundle on $E_H$ with fiber
$\mathfrak h$. The adjoint bundle ${\rm ad}(E_H)$ is defined to be the vector bundle $E_H\times^H
{\mathfrak h}$ on $M$ associated to $E_H$ for the adjoint action of $H$ on its Lie algebra
${\mathfrak h}$. From the above identification of ${\ker}(d\pi)$ with $E_H\times {\mathfrak h}$
it follows that ${\rm ad}(E_H)\,=\, {\ker}(d\pi)/H$. Given a section
$v$ of ${\rm ad}(E_H)$ defined on $U\, \subset\, M$, we shall use the same notation $v$
for its pull-back to a section of ${\ker}(d\pi) \,\subset \,T{E_H}$ on $\pi^{-1}(U)$.

\begin{definition}\label{de1}
A {\it holomorphic Cartan geometry} of type $(G,\,H)$ on $M$ is a pair $(E_H,\,\Theta)$, where
$E_H$ is a holomorphic principal $H$-bundle over $M$, and
\begin{equation}\label{x2}
\Theta\,:\,T{E_H}\,\,\xrightarrow{\,\,\,\,\sim\,\,\,}\,\, E_H \times{\mathfrak g}
\end{equation}
is a holomorphic isomorphism of vector bundles, such that
\begin{enumerate}
\item $\Theta$ is $H$--equivariant, and

\item the restriction of $\Theta$ to every fiber of $\pi$ (see \eqref{x1}) coincides with the Maurer-Cartan form
on the fiber for the action of $H$ on it.
\end{enumerate}
\end{definition}

{}From the condition in Definition \ref{de1} that $\Theta$ is an isomorphism it follows immediately
that $\dim M \,=\, \dim G - \dim H \,=\, \dim G/H$.

Consider the holomorphic principal $G$-bundle over $M$
\begin{equation}\label{x0}
E_G\,= \,E_H\times^H G \, \longrightarrow\, M
\end{equation}
obtained by extending the structure group of the principal $H$-bundle $E_H$ using the inclusion map of $H$ in 
$G$. We recall that the adjoint bundle ${\rm ad}(E_G)$ is the vector bundle $E_G\times^G{\mathfrak g}$ on $M$
associated to $E_G$ for the adjoint action of $G$ on its Lie algebra ${\mathfrak g}$. Therefore,
${\rm ad}(E_G)$ coincides with the vector bundle $E_H\times^H{\mathfrak g}$
associated to $E_H$ for the adjoint action of $H$ on ${\mathfrak g}$.

The isomorphism $\Theta$ in \eqref{x2} produces an isomorphism
\begin{equation}\label{x3}
\Theta_H\,:\, (TE_H)/H \,\,\xrightarrow{\,\,\,\,\sim\,\,\,}\,\, (E_H\times {\mathfrak g})/H\,=\, 
E_H\times^H {\mathfrak g}\,=\, {\rm ad}(E_G)
\end{equation}
of the quotients because $\Theta$ is $H$--equivariant. Recall that the Atiyah bundle $\text{At}(E_H)\,
\longrightarrow\, M$ for $E_H$ is, by definition,
\begin{equation}\label{x3c}
\text{At}(E_H)\,\, :=\,\, (TE_H)/H ,
\end{equation}
and it fits into the Atiyah exact sequence
\begin{equation}\label{x3a}
0\,\longrightarrow\, \text{ad}(E_H) \,\longrightarrow\, \text{At}(E_H) \,\longrightarrow\, TM \,\longrightarrow\, 0
\end{equation}
(see \cite{At}). From \eqref{x3} and \eqref{x3c} it follows that
the isomorphism $\Theta_H$ in \eqref{x3} is an isomorphism
\begin{equation}\label{x3b}
\Theta_H\,:\, \text{At}(E_H) \,\,\xrightarrow{\,\,\,\,\sim\,\,\,}\,\, \text{ad}(E_G).
\end{equation}

The Atiyah bundle of a holomorphic vector bundle of rank $r$ is the Atiyah bundle of the corresponding
principal $\text{GL}(r,{\mathbb C})$--bundle.

The isomorphism $\Theta_H$ in \eqref{x3} induces a holomorphic connection on $E_G$
\cite{Sh}, \cite[(2.8)]{BD}, which in turn produces a holomorphic connection on ${\rm ad}(E_G)$.
Consequently, we have a holomorphic differential operator
\begin{equation}\label{x4}
{\mathcal D}\,:\, {\rm ad}(E_G) \,\longrightarrow\, \Omega^1_M \otimes{\rm ad}(E_G)
\end{equation}
of order one.

Let $(E_H,\,\Theta)$ and $ (F_H,\,\Phi)$ be two holomorphic Cartan geometries of common
type $(G, \,H)$ on $M$.
An {\it isomorphism} $(E_H,\,\Theta)\,\,\xrightarrow{\,\,\,\sim\,\,}\,\, (F_H,\,\Phi)$ of holomorphic Cartan geometries
is a holomorphic isomorphism $\beta\,:\,E_H \,\longrightarrow\, F_H$ of principal $H$-bundles that takes $\Theta$ to
$\Phi$ so that the diagram
\begin{equation}\label{eq:iso}
	\begin{tikzcd}
		  & TE_H \arrow[r, "\Theta"] \arrow[dl] \arrow[dd, "d\beta"]&  E_H \times^H {\mathfrak g}  \arrow[dd, "\widetilde{\beta}"] \\
		  M &&\\
		& T{F_H} \arrow[ul] \arrow[r, "\Phi"] & F_H \times^H {\mathfrak g}
	\end{tikzcd}
\end{equation}
is commutative, where $\widetilde\beta$ is given by the map
$$
\beta\times{\rm Id}_{\mathfrak g}\, :\, E_H\times{\mathfrak g}\, \longrightarrow\, F_H\times{\mathfrak g}
$$
and $d\beta\,:\, TE_H\, \longrightarrow\, \beta^* TF_H$ is the differential of the map $\beta$.

\subsection{Logarithmic tangent bundle}\label{se2.2}

Let $M$ be a connected complex manifold. A reduced effective divisor $D\, \subset\, M$ is called
a \textit{normal crossing divisor} if each irreducible component of $D$ is smooth and the
irreducible components of $D$
intersect transversally. For a normal crossing divisor $D$, the logarithmic tangent bundle $TM(-\log D)$
is the subsheaf of the holomorphic tangent bundle $TM$ defined by the following condition: A holomorphic
vector field $v$ on $U\, \subset\, M$ lies in $TM(-\log D)$ if $v(f) \, \in\, H^0(U,\,
{\mathcal O}_U(-U\bigcap D))$ for all $f \, \in\, H^0(U,\, {\mathcal O}_U(-U\bigcap D))$. The
subsheaf $TM(-\log D)\, \subset\, TM$ is actually locally free, and it is closed under the Lie
bracket operation.

Consider all pairs of the form $(M,\, D)$, where $M$ is a compact complex manifold and $D\, \subset\,
M$ is a normal crossing divisor, such that the holomorphic vector bundle $TM(-\log D)$ is
holomorphically trivial. They were classified in \cite{Wi}, which is briefly recalled below.

Take any pair $(M,\, D)$ of the above type. Denote by
$\text{Aut}_D(M)$ the group of all holomorphic automorphisms of $M$ that preserve $D$.
Let $M_0\,:=\,M\setminus D$ be the complement. Denote by
\begin{equation}\label{x5}
{\mathbb A}\, \, \subset\,\, \text{Aut}_D(M)
\end{equation}
the connected component containing the identity element. This $\mathbb A$ is a finite 
dimensional connected complex Lie group. The natural action of $\mathbb A$ on $M_0$ is 
transitive. The isotropy subgroup in $\mathbb A$ of every point in $M_0$ is discrete. Let $Z$ 
denote the connected component of the center of $\mathbb A$ containing the identity element.

Let us recall now two theorems of \cite{Wi}.

\begin{theorem}[{\cite[p.~196, Theorem 1]{Wi}}]\label{th1}
There is smooth equivariant compactification $Z\, \hookrightarrow\, \overline{Z}$, a compact complex 
parallelizable manifold $B$, and a locally holomorphically trivial fibration
$$
\varpi\, :\, M\, \longrightarrow\, B\, ,
$$
such that
\begin{itemize}
\item $Z$ is a semi-torus (we recall
that a complex Lie group $C$ is called a semi-torus if it is a quotient of the additive group
$({\mathbb C}^{\dim C},\, +)$ by a discrete subgroup that generates the vector space
${\mathbb C}^{\dim C}$),

\item any isotropy subgroup of $Z$ for its action on $\overline{Z}$ is a semi-torus,

\item the typical fiber of $\varpi$ is $\overline{Z}$, and the structure group of the holomorphic fiber bundle $\varpi$
is $Z$,

\item the projection $\varpi$ is $\mathbb A$--equivariant and it admits a holomorphic connection
preserved by the action of $\mathbb A$, and

\item the quotient Lie group ${\mathbb A}/Z$ acts transitively on $B$ with discrete isotropies.
\end{itemize}
\end{theorem}

The Lie algebra of $\mathbb A$ will be denoted by $\mathfrak a$. We have
\begin{equation}\label{la}
{\mathfrak a} \,\,=\,\, H^0(M,\, TM(-\log D)).
\end{equation}

The following is a converse of Theorem \ref{th1}.

\begin{theorem}[{\cite[p.~196, Theorem 2]{Wi}}]\label{th2}
Let $B$ be a connected compact complex parallelizable manifold, $Z$ a semi-torus and $\overline{Z}$ a smooth
equivariant compactification of $Z$. Assume that all the isotropy subgroups for the $Z$--action on $\overline{Z}$
are semi-tori. Let $E$ be a holomorphic principal $Z$--bundle over $B$ admitting a holomorphic connection. Let
$\overline{E}\, :=\, E\times^Z \overline{Z}$ be the holomorphic fiber bundle over $B$ associated to $E$ for the
action of $Z$ on $\overline{Z}$. Denote the divisor $\overline{E}\setminus E\, \subset\, \overline{E}$ by $D$.
Then $D$ is a normal crossing divisor, and $T\overline{E}(-\log D)$ is holomorphically trivial.
\end{theorem}

\subsection{Homogeneous Cartan geometry}

As in Section \ref{se2.2}, $M$ is a compact complex manifold, and $D\, \subset\, M$ is
a normal crossing divisor such that $TM(-\log D)$ is holomorphically trivial.
Let $(E_H,\,\Theta)$ be a holomorphic Cartan geometry of type $(G,\,H)$ on $M$.
Take any $\rho\, \in\, {\mathbb A}$ (see \eqref{x5}). Consider the pulled back
holomorphic principal $H$--bundle $\rho^*E_H$ on $M$. There is a natural $H$--equivariant map
$$\widetilde{\rho}\, :\, \rho^*E_H\, \longrightarrow\, E_H$$ over $\rho$. Indeed, $\rho^*E_H$ is the fiber
product of the two maps $\rho\, :\, M\, \longrightarrow\, M$ and $\pi\, :\, E_H\, \longrightarrow\, M$,
and therefore, $\rho^*E_H$ is a submanifold of $M\times E_H$; the above map $\widetilde{\rho}$ is the
restriction, to $\rho^*E_H$, of the natural projection $M\times E_H\, \longrightarrow\, E_H$ to the
second factor. Let
$$
d\widetilde{\rho}\, :\, T(\rho^*E_H)\, \longrightarrow\, TE_H
$$
be the differential of the above map $\widetilde{\rho}$. The composition of maps
$$ T(\rho^*E_H)\,\, \xrightarrow{\,\,\,\, d\widetilde{\rho}\,\,\,}\,\, TE_H \,\, \xrightarrow{\,\,\,\, \Theta\,\,\,}
\,\, E_H\times{\mathfrak g} \,\, \xrightarrow{\,\,\,\, (\widetilde{\rho})^{-1}\times{\rm Id}_{\mathfrak g}\,\,\,}
\,\, \rho^*E_H\times {\mathfrak g}
$$
will be denoted by $\rho^*\Theta$. Note that $(\rho^*E_H,\, \rho^*\Theta)$ is a 
holomorphic Cartan geometry on $M$ of type $(G,\, H)$.

\begin{definition}\label{def2}
The holomorphic Cartan geometry $(E_H,\,\Theta)$ on $M$ of type $(G,\, H)$ is called \textit{weakly homogeneous}
if the holomorphic Cartan geometry $(\rho^*E_H,\, \rho^*\Theta)$ of type $(G,\, H)$ is isomorphic to $(E_H,\,
\Theta)$ for every $\rho\, \in\, {\mathbb A}$.
\end{definition}

Let ${\mathcal G}$ be a connected complex Lie group and
$$
\gamma\,\, :\,\, {\mathcal G}\,\, \longrightarrow\, {\mathbb A}
$$
a holomorphic homomorphism. Using $\gamma$, the natural action of $\mathbb A$ on $M$ produces an
action of $\mathcal G$ on $M$. More precisely, the automorphism of $M$ given by the action any
$g\, \in\, {\mathcal G}$ is the automorphism $\gamma (g)\, \in\, {\mathbb A}\, \subset\, \text{Aut}_D(M)$.
so the action of $g$ on $M$ preserves the divisor $D$.

A holomorphic principal $H$--bundle $E_H$ on $M$ is called $\gamma$--\textit{homogeneous} if $E_H$
is equipped with a holomorphic action of $\mathcal G$ such that
\begin{itemize}
\item the actions of $\mathcal G$ and $H$ on $E_H$ commute, and

\item the natural projection $\pi\, :\, E_H\, \longrightarrow\, M$ (see \eqref{x1}) is $\mathcal G$--equivariant
(it was noted above that $\mathcal G$ acts on $M$).
\end{itemize}

Let $E_H$ be a $\gamma$--\textit{homogeneous} principal $H$--bundle on $M$. The action of $\mathcal G$ on
$E_H$ produces an action of $\mathcal G$ on $TE_H$. Take a homomorphism
$\Theta\,:\,T{E_H}\,\longrightarrow\, E_H \times{\mathfrak g}$ as in \eqref{x2}. The action of $\mathcal G$
on $E_H$ and the trivial action of $\mathcal G$ on $\mathfrak g$ together produce an action of $\mathcal G$ on
$E_H\times{\mathfrak g}$. The homomorphism $\Theta$ will be called
$\gamma$--\textit{invariant} if it is $\mathcal G$--equivariant.

\begin{definition}\label{def-h}
A holomorphic Cartan geometry $(E_H,\,\Theta)$ on $M$ of type $(G,\, H)$ is called
$\gamma$--\textit{homogeneous} if the principal
$H$--bundle $E_H$ is $\gamma$--homogeneous and the homomorphism $\Theta$ is $\gamma$--invariant.
\end{definition}

\begin{remark}\label{rem1}
Consider the identity map ${\rm Id}_{\mathbb A}\, :\, {\mathbb A}\, \longrightarrow\,{\mathbb A}$.
If $(E_H,\, \Theta)$ is an ${\rm Id}_{\mathbb A}$--homogeneous holomorphic Cartan geometry on $M$ of type $(G,\, H)$, then clearly
$(E_H,\, \Theta)$ is weakly homogeneous. However, a weakly homogeneous bundle need not be ${\rm Id}_{\mathbb A}$--homogeneous.
For example, the pullback of the tautological line bundle ${\mathcal O}_{{\mathbb C}{\mathbb P}^n}(1)$ on
${\mathbb C}{\mathbb P}^n$, by any holomorphic automorphism of ${\mathbb C}{\mathbb P}^n$, is isomorphic
to ${\mathcal O}_{{\mathbb C}{\mathbb P}^n}(1)$. But the action of $\text{Aut}({\mathbb C}{\mathbb P}^n)
\,=\,\text{PGL}(n+1, {\mathbb C})$ on ${\mathbb C}{\mathbb P}^n$ does not lift to
${\mathcal O}_{{\mathbb C}{\mathbb P}^n}(1)$. Note that the action of
$\text{PGL}(n+1, {\mathbb C})$ on ${\mathbb C}{\mathbb P}^n$ lifts to ${\mathcal O}_{{\mathbb C}{\mathbb P}^n}(1)$
if and only if the action of $\text{PGL}(n+1, {\mathbb C})$ on ${\mathbb C}{\mathbb P}^n$ lifts to
$H^0({\mathbb C}{\mathbb P}^n,\, {\mathcal O}_{{\mathbb C}{\mathbb P}^n}(1))\ =\, {\mathbb C}^{n+1}$;
we know that the action of $\text{PGL}(n+1, {\mathbb C})$ on ${\mathbb C}{\mathbb P}^n$ does
not lift to ${\mathbb C}^{n+1}$.
\end{remark}

\section{Weakly homogeneous Cartan geometry}

\subsection{Automorphisms of a Cartan geometry}\label{se3.1}

Let $(E_H,\,\Theta)$ be a weakly homogeneous Cartan geometry of type $(G,\, H)$ on $M$. Take
any $\rho\, \in\, \mathbb A$. Note that giving a holomorphic isomorphism
$$
E_H\,\, \longrightarrow\,\, \rho^*E_H
$$
of principal $H$--bundles is equivalent to giving a holomorphic isomorphism
$$
\eta\,\,:\,\, E_H\,\, \longrightarrow\,\, E_H
$$
such that
\begin{itemize}
\item $\eta(zh)\,=\, \eta(z)h$ for all $z\, \in\, E_H$ and $h\, \in\, H$, and

\item $\pi\circ\eta (z)\,=\, \rho(\pi(z))$ for all $z\, \in\, E_H$, where $\pi$ is the projection
$E_H\, \longrightarrow\, M$ as in \eqref{x1}.
\end{itemize}

Denote by $\widetilde{\Gamma}$ the space of all pairs $(\rho,\, \eta)$, where $\rho\, \in\, {\mathbb A}$, and
$\eta\,:\, E_H\, \longrightarrow\, E_H$ is a holomorphic isomorphism  such that
\begin{itemize}
\item $\eta(zh)\,=\, \eta(z)h$ for all $z\, \in\, E_H$ and $h\, \in\, H$,

\item $\pi\circ\eta (z)\,=\, \rho(\pi(z))$ for all $z\, \in\, E_H$, and

\item $\eta$ preserve $\Theta$, meaning the composition of homomorphisms
$$
TE_H \,\, \xrightarrow{\,\,\,d\eta\,\,\,}\,\, TE_H \,\, \xrightarrow{\,\,\,\Theta\,\,\,}\,\, E_H\times {\mathfrak g}
\,\, \xrightarrow{\,\,\,\eta^{-1}\times {\rm Id}_{\mathfrak g}\,\,\,}\,\, E_H\times {\mathfrak g}
$$
coincides with $\Theta$, where $d\eta\, :\, TE_H\, \longrightarrow\, \eta^*TE_H$ is the differential of
the map $\eta$.
\end{itemize}
Since $(E_H,\,\Theta)$ is weakly homogeneous, for every $\rho\, \in\, {\mathbb A}$ there is an isomorphism $\eta$
satisfying the above conditions. This $\widetilde\Gamma$ is a group. Indeed,
$$(\rho_1,\, \eta_1)\cdot (\rho_2,\, \eta_2)\,\,:=\,\, (\rho_1\circ\rho_2,\, \eta_1\circ\eta_2)$$
is a group operation on $\widetilde\Gamma$.

It can be shown that $\widetilde\Gamma$ is a finite dimensional complex Lie group. To see this,
let $$\beta_0\, :\, \widetilde{\Gamma}\, \longrightarrow\, {\mathbb A}$$
be the homomorphism that sends any pair $(\rho,\, \eta)$ as above to $\rho$. The kernel of $\beta_0$
will be denoted by $\text{Aut}_\Theta(E_H)$. So
\begin{equation}\label{x8}
\text{Aut}_\Theta(E_H)\,\,:=\,\, \{(\rho,\, \eta)\,\in\, \widetilde{\Gamma}\,\,\big\vert\,\, \rho\,=\, {\rm Id}_M\}\,\,
\subset\, \widetilde{\Gamma}
\end{equation}
is the group of all holomorphic automorphisms of the principal $H$--bundle $E_H$ that preserve $\Theta$.
Since ${\mathbb A}$ is a finite dimensional complex Lie group,
to prove that $\widetilde\Gamma$ is a finite dimensional complex Lie group it suffices to show that
$\text{Aut}_\Theta(E_H)$ is a finite dimensional complex Lie group. Let
$$\text{Ad}(E_H)\, =\, E_H\times^H H\, \longrightarrow\, M$$
be the holomorphic fiber bundle over $M$ associated to $E_H$ for the adjoint action of $H$ on itself.
So each fiber of $\text{Ad}(E_H)$ is a group isomorphic to $H$. Let
$$\text{ad}(E_H)\, =\, E_H\times^H H\, \longrightarrow\, M$$
be the holomorphic fiber bundle over $M$ associated to $E_H$ for the adjoint action of $H$ on its Lie algebra
$\mathfrak h$. So $\text{ad}(E_H)$ is a holomorphic vector bundle on $M$ whose every fiber is a Lie algebra isomorphic
to $\mathfrak h$. Clearly, $\text{ad}(E_H)$ is the Lie algebra bundle for the bundle $\text{Ad}(E_H)$ of
Lie groups. We note that $\text{Aut}_\Theta(E_H)$ is a closed subgroup of the space of holomorphic sections of
$\text{Ad}(E_H)$. Since $M$ is compact, the space of holomorphic sections of
$\text{Ad}(E_H)$ is a finite dimensional complex Lie group; it's Lie algebra is actually
$H^0(M,\, \text{ad}(E_H))$. Hence $\text{Aut}_\Theta(E_H)$ is a finite dimensional complex Lie group.

Let
\begin{equation}\label{x6}
\Gamma\,\, \subset\,\, \widetilde{\Gamma} 
\end{equation}
be the connected component containing the identity element. It fits in a short exact sequence of complex Lie groups
\begin{equation}\label{x7}
1\, \longrightarrow\, \text{Aut}^0_\Theta(E_H)\, \longrightarrow\, \Gamma\,
\stackrel{\tau}{\longrightarrow}\, {\mathbb A}\,\longrightarrow\, 1,
\end{equation}
where $\text{Aut}^0_\Theta(E_H)\, =\, \text{Aut}_\Theta(E_H)\bigcap \Gamma$ (see \eqref{x8}), so $\text{Aut}_\Theta(E_H)/
\text{Aut}^0_\Theta(E_H)$ is a discrete set, and the projection $\tau$ in \eqref{x7} sends any
$(\rho,\, \eta)\, \in\, \Gamma$ to $\rho$. Note that $\tau$ is surjective because
$(E_H, \Theta)$ is weakly homogeneous. The converse statement --- that $(E_H, \Theta)$ is weakly homogeneous
if $\tau$ is surjective --- is evidently true. So, $(E_H, \Theta)$ is weakly homogeneous
if and only if $\tau$ is surjective.

{}From \eqref{x3c} it follows immediately that
\begin{equation}\label{x10}
H^0(M,\, \text{At}(E_H))\,=\, H^0(E_H,\, TE_H)^H\, \subset\, H^0(E_H,\, TE_H),
\end{equation}
where $H^0(E_H,\, TE_H)^H$ denotes the space of $H$--invariant holomorphic vector fields on $E_H$.

{}From Definition \ref{de1} we know that $\Theta$ is a $\mathfrak g$--valued holomorphic $1$--form on $E_H$ satisfying
certain conditions. Now, for any $\mathfrak g$--valued holomorphic $1$--form $\Psi
\, \in\, H^0(E_H,\, \Omega^1_{E_H}\otimes {\mathfrak g})$ on $E_H$, and any holomorphic vector
field $\gamma\,\in\, H^0(M,\, \text{At}(E_H))$, we define the Lie derivative
$$
L_{\gamma} \Psi\, \in\, H^0(E_H,\, \Omega^1_{E_H}\otimes {\mathfrak g})
$$
as follows:
\begin{equation}\label{x9}
(L_\gamma \Psi)(v)\,\,=\,\, \gamma(\Psi (v)) - \Psi([\gamma,\, v]),
\end{equation}
for all locally defined holomorphic vector fields $v$ on $E_H$; note that all the three terms in \eqref{x9} are locally
defined holomorphic sections of the trivial vector bundle $E_H\times {\mathfrak g}\, \longrightarrow\, E_H$.

\begin{proposition}\label{prop1}
The Lie algebra ${\rm Lie}(\Gamma)$ of the complex Lie group $\Gamma$ in \eqref{x6} is the subspace of
$H^0(E_H,\, TE_H)$ consisting of all $\gamma\, \in\, H^0(E_H,\, TE_H)$ satisfying the following conditions:
\begin{enumerate}
\item $\gamma\, \in\, H^0(E_H,\, TE_H)^H\, \subset\, H^0(E_H,\, TE_H)$,

\item $L_\gamma\Theta\,=\, 0$, where $L_\gamma$ in defined in \eqref{x9}, and

\item $d\pi (\gamma)\, \in\, H^0(M, \, TM(-\log D))$, where $d\pi$ is the homomorphism in \eqref{x-1}.
\end{enumerate}
The Lie bracket operation of ${\rm Lie}(\Gamma)$ is given by the Lie bracket operation of
vector fields on $E_H$.
\end{proposition}

\begin{proof}
It is evident that the subspace of $H^0(E_H,\, TE_H)$ satisfying the above three conditions is closed 
under the Lie bracket operation of vector fields. If fact, the subspace of $H^0(E_H,\, TE_H)$ satisfying 
the first condition is closed under the Lie bracket operation of vector fields. Similarly, the subspace 
of $H^0(E_H,\, TE_H)$ satisfying the second condition is closed under the Lie bracket operation of 
vector fields. The same holds for the third condition.

The first condition ensures that $\gamma\, \in\, H^0(M,\, \text{At}(E_H))$ (see \eqref{x3c}). The third 
condition is equivalent to the condition that
\begin{equation}\label{x12}
d\pi (\gamma)\,\in \, H^0(M,\, TM(-\log D)) \,=\, \text{Lie}({\mathbb A})\,=\, {\mathfrak a}
\end{equation}
(see \eqref{x5}, \eqref{la}).
The second condition is equivalent to the condition that the flow generated by $\gamma$ preserves
$\Theta$. The proposition follows from these.
\end{proof}

\subsection{Existence of a logarithmic connection}

Let
$$
{\mathcal D}\,\,:=\,\, \pi^{-1}(D)\,\, \subset\,\, E_H
$$
be the inverse image of $D$, where $\pi$ is the projection in \eqref{x1}. Let
$$
TE_H(-\log {\mathcal D})\,\, \subset\,\, TE_H
$$
be the logarithmic tangent bundle. Note that
\begin{equation}\label{x13}
TE_H(-\log {\mathcal D})\,\, =\,\, (d\pi)^{-1}(\pi^* (TM(-\log D))),
\end{equation}
where $d\pi$ is the homomorphism in \eqref{x-1}. The action of $H$ on $TE_H$ actually preserves
the subsheaf $TE_H(-\log {\mathcal D})$. Indeed, this follows immediately from \eqref{x13}
and the fact that the projection $d\pi \,:\, T{E_H} \,\longrightarrow\, \pi^* TM$ in \eqref{x-1}
is $H$--equivariant. Note that the third condition in Proposition \ref{prop1}, asserting that
$d\pi (\gamma)\, \in\,H^0(M, \, TM(-\log D))$, is equivalent to
the condition that $\gamma\, \in\, H^0(E_H, \, TE_H(-\log {\mathcal D}))$.

Define
$$
\text{At}(E_H)(-\log D) \,\,:=\,\, TE_H(-\log {\mathcal D})/H,
$$
which is a vector bundle on $M$. Recall that that $\text{ad}(E_H)\,=\,\ker\, d\pi/H$,
and hence we have $\text{ad}(E_H)\,\subset\, \text{At}(E_H)(-\log D)$. From \eqref{x3a} and \eqref{x13} we have the short exact sequence
\begin{equation}\label{x14}
0\,\longrightarrow\, \text{ad}(E_H) \,\longrightarrow\, \text{At}(E_H)(-\log D) \,\stackrel{\varphi}{\longrightarrow}\, TM(-\log D)
\,\longrightarrow\, 0.
\end{equation}
Let
\begin{equation}\label{x15}
0\,\longrightarrow\, H^0(M,\, \text{ad}(E_H)) \,\longrightarrow\, H^0(M,\, \text{At}(E_H)(-\log D))
\end{equation}
$$
\stackrel{h_1}{\longrightarrow}\, H^0(M,\, TM(-\log D))\,\stackrel{h_2}{\longrightarrow}\, H^1(M,\, \text{ad}(E_H))
$$
be the long exact sequence of cohomologies associated to \eqref{x14}.

\begin{lemma}\label{lem1}
The homomorphism $h_1$ in \eqref{x15} is surjective. In other words, $h_2$ in \eqref{x15} is the zero homomorphism.
\end{lemma}

\begin{proof}
{}From \eqref{x13} and \eqref{x3c} it follows that $H^0(M,\, \text{At}(E_H)(-\log D))$ is identified with the 
subspace of $H^0(E_H,\, TE_H)$ consisting of all $\gamma\, \in\, H^0(E_H,\, TE_H)$ satisfying the following conditions:
\begin{enumerate}
\item $\gamma\, \in\, H^0(E_H,\, TE_H)^H\, \subset\, H^0(E_H,\, TE_H)$, and

\item $d\pi (\gamma)\, \in\, H^0(M, \, \gamma^*TM(-\log D))$, where $d\pi$ is the homomorphism in \eqref{x-1}.
\end{enumerate}
Consequently, from Proposition \ref{prop1} we conclude that
\begin{equation}\label{x16}
{\rm Lie}(\Gamma)\,\, \subset\,\, H^0(M,\, \text{At}(E_H)(-\log D)).
\end{equation}

Consider the projection $h_1$ in \eqref{x15}. Recall from \eqref{x12} that $H^0(M,\, TM(-\log D))\,=\, 
{\mathfrak a}$. The restriction of $h_1$ (see \eqref{x15}) to the subspace ${\rm Lie}(\Gamma)$ in 
\eqref{x16} coincides with the homomorphism of Lie algebras associated to the projection $\Gamma\, 
\longrightarrow\, {\mathbb A}$ in \eqref{x7}. Since this homomorphism
$\Gamma\, \longrightarrow\, {\mathbb A}$ is surjective, it follows immediately that the restriction of 
$h_1$ to the subspace ${\rm Lie}(\Gamma)$ in \eqref{x16} is surjective. Hence $h_1$ is surjective. This 
implies that $h_2\,=\, 0$.
\end{proof}

{}From Lemma \ref{lem1} and \eqref{x15} we have the short exact sequence
\begin{equation}\label{x17}
0\,\longrightarrow\, H^0(M,\, \text{ad}(E_H)) \,\longrightarrow\, H^0(M,\, \text{At}(E_H)(-\log D))
\, \stackrel{h_1}{\longrightarrow}\, H^0(M,\, TM(-\log D))\,\longrightarrow\, 0.
\end{equation}

Recall that holomorphic sections of $\text{ad}(E_H)$ are precisely the $H$--invariant holomorphic vertical vector fields on
$E_H$ for the projection $\pi$ in \eqref{x1}. Define
\begin{equation}\label{x18}
{\mathcal V}_\Theta\,:=\, \{\gamma\, \in\, H^0(E_H,\, TE_H)^H\,\big\vert\, d\pi(\gamma)\,=\, 0\ \text{ and}\
L_\gamma\Theta\,=\, 0\}\, \subset\, H^0(M,\, \text{ad}(E_H));
\end{equation}
$d\pi$ and $L_\gamma\Theta$ are defined in \eqref{x-1} and \eqref{x9} respectively. Note that ${\mathcal V}_\Theta$
is the Lie algebra of the complex Lie group $\text{Aut}_\Theta(E_H)$ in \eqref{x8}. Indeed, we already  proved  that  $\text{Aut}_\Theta(E_H)$  is a complex Lie group whose Lie  algebra coincides with the subalgebra in 
$H^0(M,\, \text{ad}(E_H))$, given by those vector fields in $H^0(E_H,\, TE_H)^H$ which preserve $\Theta$.  This subalgebra is precisely ${\mathcal V}_\Theta$.

It was noted in the proof of Lemma \ref{lem1} that the restriction of $h_1$ to the subspace ${\rm Lie}(\Gamma)$ in \eqref{x16}
coincides with the homomorphism of Lie algebras associated to the projection $\Gamma\, \longrightarrow\, {\mathbb A}$ in \eqref{x7}.
Therefore, from Proposition \ref{prop1}, \eqref{x7}, \eqref{x18} and \eqref{x17} we have the following commutative diagram
\begin{equation}\label{x19}
	\begin{tikzcd}
	0\arrow[r]&\mathcal{V}_\Theta \arrow[r]\arrow[d] & {\rm Lie}(\Gamma)\arrow[r] \arrow[d] & 
	H^0(M,\, TM(-\log D)) \arrow[r] \arrow[d, "{\rm ID}"] & 0\\
	0 \arrow[r] & H^0(M,\, \text{ad}(E_H))\arrow[r] & H^0(M,\, \text{At}(E_H)(-\log D))\arrow[r, "h_1"] & H^0(M,\, TM(-\log D))\arrow[r] & 0
\end{tikzcd}
	\end{equation}
whose all the vertical arrows are injective.

Recall that {\it a logarithmic connection} on $E_H$ singular on $D$ is a holomorphic homomorphism
$$
\Delta\,\,:\,\, TM(-\log D)\,\, \longrightarrow\,\,\text{At}(E_H)(-\log D)
$$
such that $\varphi\circ\Delta \,=\, {\rm Id}_{TM(-\log D)}$, where $\varphi$ is the projection
in \eqref{x14} \cite{De}.

\begin{definition}\label{def5}
Take any $\mathfrak g$--valued holomorphic $1$--form $\Psi \, \in\, H^0(E_H,\, \Omega^1_{E_H}\otimes {\mathfrak g})$
on $E_H$ and any logarithmic connection $\Delta$ on $E_H$ singular on $D$. The logarithmic connection $\Delta$ is
said to \textit{preserve} $\Psi$ if
$$
L_{\Delta (v)}\Psi \,\,=\,\, 0
$$
for any locally defined holomorphic section $v$ of $TM(-\log D)$ (see \eqref{x9}).
\end{definition}

\begin{proposition}\label{prop2}
Let $(E_H,\,\Theta)$ be a weakly homogeneous Cartan geometry of type $(G,\, H)$ on $M$.
Then the principal $H$--bundle $E_H$ admits a logarithmic connection $\Delta$ singular on $D$ such that
$\Theta$ is preserved by $\Delta$.
\end{proposition}

\begin{proof}
Fix a $\mathbb C$--linear homomorphism
\begin{equation}\label{lh}
\Delta'\,:\, H^0(M,\, TM(-\log D))\,\, \longrightarrow\,\, H^0(M,\, \text{At}(E_H)(-\log D))
\end{equation}
such that $h_1\circ\Delta'\,=\, {\rm Id}_{H^0(M, TM(-\log D))}$, where $h_1$ is the homomorphism
in \eqref{x19}. Since $TM(-\log D)$ is holomorphically trivial, for any $w\, \in\, TM(-\log D)_x$, $x\, \in\,
M$, there is a unique $\widetilde{w}\, \in\, H^0(M,\, TM(-\log D))$ such that $\widetilde{w}(x)\,=\, w$.
The homomorphism $\Delta'$ in \eqref{lh} produces a homomorphism
\begin{equation}\label{lh2}
\Delta\,\,:\,\, TM(-\log D) \,\, \longrightarrow\, \, \text{At}(E_H)(-\log D)
\end{equation}
that sends any $w\, \in\, TM(-\log D)_x$, $x\, \in\, M$, to $\Delta'(\widetilde{w})(x)\, \in\, \text{At}(E_H)(-\log D)_x$,
where $\widetilde{w}$ is constructed as above from $w$ and $\Delta'$ is the homomorphism
in \eqref{lh}. The property that
$$\varphi\circ\Delta \,=\, {\rm Id}_{TM(-\log D)},$$ where $\varphi$ is the projection in \eqref{x14}, 
follows immediately from the fact that $h_1\circ\Delta'\,=\, {\rm Id}_{H^0(M, TM(-\log D))}$. Consequently,
$\Delta$ defines a logarithmic connection on $E_H$ singular on the divisor $D$.

To prove that   the logarithmic  connection $\Delta$ can be chosen such that   $\Theta$ is preserved, take   $\Delta'$ in \eqref{lh} such that
$$
\Delta'(H^0(M,\, TM(-\log D)))\,\, \subset\, \,{\rm Lie}(\Gamma)\,\, \subset\,\, 
H^0(M,\, \text{At}(E_H)(-\log D))
$$
(see \eqref{x19}). Then the corresponding logarithmic connection $\Delta$ on $E_H$ in \eqref{lh2} 
preserves $\Theta$.  
\end{proof}

In the next subsection we will prove a converse of Proposition \ref{prop2}

\subsection{Connection implies weak homogeneity}

\begin{proposition}\label{prop4}
Let $(E_H,\,\Theta)$ be a holomorphic Cartan geometry of type $(G,\, H)$ on $M$ satisfying the
condition that
the principal $H$--bundle $E_H$ admits a logarithmic connection $\Delta$ singular on $D$ such
that $\Theta$ is preserved by $\Delta$. Then the Cartan geometry $(E_H,\,\Theta)$ is weakly homogeneous.
\end{proposition}

\begin{proof}
Denote by $\widetilde{\Gamma}$ the space of all pairs $(\rho,\, \eta)$, where $\rho\, \in\, {\mathbb A}$, and
$\eta\,:\, E_H\, \longrightarrow\, E_H$ is a holomorphic isomorphic such that
\begin{itemize}
\item $\eta(zh)\,=\, \eta(z)h$ for all $z\, \in\, E_H$ and $h\, \in\, H$,

\item $\pi\circ\eta (z)\,=\, \rho(\pi(z))$ for all $z\, \in\, E_H$, and

\item $\eta$ preserve $\Theta$, meaning the composition of homomorphisms
$$
TE_H \,\, \xrightarrow{\,\,\,d\eta\,\,\,}\,\, TE_H \,\, \xrightarrow{\,\,\,\Theta\,\,\,}\,\, E_H\times {\mathfrak g}
$$
coincides with $\Theta$, where $d\eta\, :\, TE_H\, \longrightarrow\, \eta^*TE_H$ is the differential of $\eta$.
\end{itemize}
Our aim is to show that the natural projections
\begin{equation}\label{z1}
\widetilde{\tau}\, :\, \widetilde{\Gamma}\, \longrightarrow\, {\mathbb A},\ \ \, (\rho,\, \eta)\, \longrightarrow\,\rho
\end{equation}
is surjective. We note that $\widetilde{\Gamma}$ is a complex Lie group; this can be shown by
repeating the argument in Section \ref{se3.1}. Let
$$
{\Gamma}\,\, \subset\,\, \widetilde{\Gamma}
$$
be the connected component containing the identity element. Let
\begin{equation}\label{z2}
{\tau}\, :\, {\Gamma}\, \longrightarrow\, {\mathbb A},\ \ \, (\rho,\, \eta)\, \longrightarrow\,\rho
\end{equation}
be the restriction of the map $\widetilde{\tau}$ in \eqref{z1}. Since $\mathbb A$ is connected the homomorphism $\tau$ in \eqref{z2}
is surjective if and only if $\widetilde{\tau}$ is surjective.

Let
\begin{equation}\label{z3}
d{\tau}\, :\, \text{Lie}({\Gamma})\, \longrightarrow\, {\mathfrak a}
\end{equation}
be the homomorphism of Lie algebras corresponding to the homomorphism $\tau$ in \eqref{z2}. To prove that $\tau$ is surjective
it suffices to show that the homomorphism $d\tau$ in \eqref{z3} is surjective.

The Lie algebra $\text{Lie}({\Gamma})$ of the Lie group ${\Gamma}$ in \eqref{z2} is the subspace of
$H^0(E_H,\, TE_H)$ consisting of all $\gamma\, \in\, H^0(E_H,\, TE_H)$ satisfying the following conditions:
\begin{enumerate}
\item $\gamma\, \in\, H^0(E_H,\, TE_H)^H\, \subset\, H^0(E_H,\, TE_H)$,

\item $L_\gamma\Theta\,=\, 0$, where $L_\gamma$ in defined in \eqref{x9}, and

\item $d\pi (\gamma)\, \in\, H^0(M, \, TM(-\log D))$, where $d\pi$ is the homomorphism in \eqref{x-1}.
\end{enumerate}
Indeed, this follows immediately from the proof of Proposition \ref{prop1}. We note that
the Lie bracket operation of ${\rm Lie}(\Gamma)$ is given by the Lie bracket operation of vector fields.

Let
\begin{equation}\label{z6}
\Delta\, :\, TM(-\log D)\,\, \longrightarrow\,\,\text{At}(E_H)(-\log D)
\end{equation}
be a logarithmic connection on $E_H$ singular on $D$ such that $\Theta$ is preserved by $\Delta$.
Let
\begin{equation}\label{z4}
\Delta'\,\, :\,\, H^0(M,\, TM(-\log D)) \,=\, {\mathfrak a}\,\, \longrightarrow\,\, H^0(M,\, \text{At}(E_H)(-\log D))
\end{equation}
be the homomorphism of global sections produced by the homomorphism $\Delta$ in \eqref{z6}. In view of the above description of
${\rm Lie}(\Gamma)$, the given condition that $\Theta$ is preserved by $\Delta$ implies that
\begin{equation}\label{z5}
\Delta'(H^0(M,\, TM(-\log D)))\,\,\subset\,\, {\rm Lie}(\Gamma)\,\,\subset\, \, H^0(M,\, \text{At}(E_H)(-\log D)),
\end{equation}
where $\Delta'$ is the homomorphism in \eqref{z4}. Since $\Delta$ in \eqref{z6} gives a splitting of the logarithmic
Atiyah exact sequence for $E_H$, we conclude that
\begin{equation}\label{z7}
(d\tau) \circ \Delta'\,,=\,\, {\rm Id}_{H^0(M,\, TM(-\log D))}.
\end{equation}
{}From \eqref{z7} it follows immediately that $d\tau$ is surjective. As noted before, this completes the proof.
\end{proof}

Proposition \ref{prop2} and Proposition \ref{prop4} together give the following:

\begin{theorem}\label{thm2}
A holomorphic Cartan geometry $(E_H,\,\Theta)$ of type $(G,\, H)$ on $M$ is weakly homogeneous if and only if
the principal $H$--bundle $E_H$ admits a logarithmic connection $\Delta$ singular on $D$ such
that $\Theta$ is preserved by the logarithmic connection $\Delta$.
\end{theorem}

\begin{example}\label{ex1}
Take the simplest case of Cartan geometry, namely projective structure on $\mathbb{C}
\mathbb{P}^1$. Consider the standard action of $\mathbb{C}^*\,=\,
{\mathbb C}\setminus\{0\}$ on $\mathbb{C}\mathbb{P}^1$, so 
$0$ and $\infty$ are the only fixed points. So $M\,=\, \mathbb{C}\mathbb{P}^1$, and
$D\,=\, \{0,\, \infty\}$. The unique projective structure on $\mathbb{C}\mathbb{P}^1$
is of course preserved by this action. So Theorem \ref{thm2} gives a logarithmic connection
on the principal $H$-bundle $E_H$ (here $H$ is the Borel subgroup of $\mathrm{PGL}(2,C)$).
Note that $E_H$ does not admit any holomorphic connection. In fact, $\text{ad}(E_H)$ is the unique
nontrivial extension of ${\mathcal O}_{\mathbb{C}\mathbb{P}^1}$ by $K_{\mathbb{C}\mathbb{P}^1}\,=\,
{\mathcal O}_{\mathbb{C}\mathbb{P}^1}(-2)$. So $\text{ad}(E_H)$ does not admit any holomorphic connection.
Therefore, the logarithmic connection on the principal $H$-bundle $E_H$ given by Theorem \ref{thm2}
is a nontrivial logarithmic connection.
\end{example}

\subsection{Criterion for Homogeneous Cartan geometry}

Let $(E_H,\,\Theta)$ be a weakly homogeneous Cartan geometry of type $(G,\, H)$ on $M$.
Consider the projection $\tau$ in \eqref{x7}. We note that $(E_H,\,\Theta)$ has a tautological
$\tau$--homogeneous structure; see Definition \ref{def-h}. Indeed, any $\underline{z}\,=\, (\rho,\, \eta)
\, \in\, \Gamma$ acts on $E_H$ via $\eta\,=\, \tau(\underline{z})$.

\begin{lemma}\label{lem2}
Let ${\mathcal G}$ be a connected complex Lie group and
$$
\gamma\,\, :\,\, {\mathcal G}\,\, \longrightarrow\, {\mathbb A}
$$
a holomorphic homomorphism.
Let $(E_H,\,\Theta)$ be a weakly homogeneous Cartan geometry of type $(G,\, H)$ on $M$.
Giving a $\gamma$--homogeneous structure on $(E_H,\,\Theta)$ is equivalent to giving a holomorphic
homomorphism of Lie groups
$$
\beta\,\, :\,\, {\mathcal G}\,\, \longrightarrow\,\, \Gamma
$$
such that $\tau\circ\beta\,=\, \gamma$, where $\tau$ is the projection in \eqref{x7}.
\end{lemma}

\begin{proof}
First assume that $(E_H,\,\Theta)$ has a $\gamma$--homogeneous structure. For any
$g\, \in\, {\mathcal G}$, consider the automorphism of $E_H$ over $\gamma (g)\, \in\,
\text{Aut}_D(M)$ given by the action of $g$ on $E_H$. This
automorphism of $E_H$ over $\gamma (g)$ is evidently contained in $\Gamma$. Therefore,
we get a holomorphic homomorphism
$$
\beta\,\, :\,\, {\mathcal G}\,\, \longrightarrow\,\, \Gamma
$$
that sends any $g\, \in\, {\mathcal G}$ to the automorphism of $E_H$ over $\gamma (g)$
given by the action of $g$ on $E_H$. Clearly, we have $\tau\circ\beta\,=\, \gamma$.

Conversely, if
$$
\beta\,\, :\,\, {\mathcal G}\,\, \longrightarrow\,\, \Gamma
$$
is a holomorphic homomorphism with $\tau\circ\beta\,=\, \gamma$, then associating to
any $g\, \in\, {\mathcal G}$ the automorphism of $\beta(g)$ of $E_H$ a
$\gamma$--homogeneous structure on $(E_H,\,\Theta)$ is obtained.
\end{proof}

\section{Logarithmic Cartan geometry}\label{section Log Cartan}

Take $(G,\, H)$ as before. Fix a pair $(V,\, \chi)$, where $V$ is a finite dimensional complex vector space, and
\begin{equation}\label{chi}
\chi\, :\, G\, \longrightarrow\, \text{GL}(V)
\end{equation}
is a holomorphic homomorphism with discrete kernel. Note that $\ker(\chi)$ is
discrete if and only if the homomorphism of Lie algebras corresponding
to $\chi$
\begin{equation}\label{g0}
d\chi\, :\, {\rm Lie}(G)\, =\, {\mathfrak g}\, \longrightarrow\, \text{Lie}(\text{GL}(V))\,=\, \text{End}(V)
\end{equation}
is injective. Such a homomorphism $\chi$ exists if $G$ simply connected (by Ado's
Theorem). The restriction $\chi\big\vert_H$ of $\chi$ to the subgroup $H\, \subset\, G$ will
be denoted by $\chi_H$.

Take $(M,\, D)$ as before. Let $E'_H$ be a holomorphic principal $H$--bundle on
$M'\, :=\, M\setminus D$ and
\begin{equation}\label{g1}
\Theta'\, :\, TE'_H\, \stackrel{\sim}{\longrightarrow}\, E'_H\times {\mathfrak g}
\end{equation}
a holomorphic Cartan geometry on $M'$ of type $(G,\, H)$. Denote by $E^V_H$ the
holomorphic vector bundle on $M'$ associated to $E'_H$ for the homomorphism $\chi_H\,=\, \chi\big\vert_H$.
Also, denote by $E'_H(V)\,=\, E'_H(\text{GL}(V))$ the
holomorphic principal $\text{GL}(V)$--bundle on $M'$ associated to $E'_H$ for the
homomorphism $\chi_H$. So $E^V_H$ is identified with the holomorphic vector bundle on
$M'$ associated to the principal $\text{GL}(V)$--bundle $E'_H(V)$ for the standard action
of $\text{GL}(V)$ on $V$. The isomorphism $\Theta'$ in \eqref{g1} and the homomorphism
$d\chi$ in \eqref{g0} together produce a holomorphic connection on the principal $\text{GL}
(V)$--bundle $E'_H(V)$ \cite[Lemma 3.1]{BDM}; this holomorphic connection on $E'_H(V)$ will
be denoted by ${\mathcal D}'$. The holomorphic connection on $E^V_H$ induced
by ${\mathcal D}'$ will also be denoted by ${\mathcal D}'$.

\begin{definition}\label{def3}
A \textit{logarithmic Cartan geometry} of type $(G,\, H)$ on $(M,\, D)$ is a
holomorphic Cartan geometry $(E'_H,\, \Theta')$ of type $(G,\, H)$ on $M'$ such that
$E^V_H$ extends to a holomorphic vector bundle $\widehat{E}^V_H$ on $M$ satisfying the
condition that the connection ${\mathcal D}'$ on $E^V_H$ is a logarithmic
connection on $\widehat{E}^V_H$. (See \cite[Definition 3.2]{BDM}, \cite[Lemma 3.3]{BDM}.)
\end{definition}

Note that Definition \ref{def3} depends on the choice of the pair $(V,\, \chi)$.

The above connection ${\mathcal D}'$ on the principal $\text{GL}(V)$--bundle $E'_H(V)$
is given by a holomorphic homomorphism of vector bundles over $E'_H(V)$
$$
\Psi\,:\, TE'_H(V)\, \longrightarrow\,E'_H(V)\times \text{End}(V)
$$
such that
\begin{enumerate}
\item $\Psi$ is $\text{GL}(V)$--equivariant (the action of $\text{GL}(V)$ on
$TE'_H(V)$ is given by the action of $\text{GL}(V)$ on $E'_H(V)$ and $\text{GL}(V)$ acts
diagonally on $E'_H(V)\times \text{End}(V)$ using the adjoint action), and

\item $\Psi$ coincides with the Maurer--Cartan form when restricted to any fiber of the
bundle $E'_H(V)$.
\end{enumerate}

The conditions in Definition \ref{def3} that
$E^V_H$ extends to a holomorphic vector bundle $\widehat{E}^V_H$ on $M$ 
such that the connection ${\mathcal D}'$ on $\widehat{E}^V_H$ is a logarithmic
connection on $\widehat{E}^V_H$ is equivalent to the following:
\begin{itemize}
\item The principal $\text{GL}(V)$--bundle $E'_H(V)$ on $M'$ extends to a holomorphic
principal $\text{GL}(V)$--bundle
$$
q_0\,\, :\,\, \widehat{E}_H(V)\,\, \longrightarrow\,\, M$$
on $M$, and

\item $\Psi$ extends to a holomorphic homomorphism
$$
T\widehat{E}_H(V)(-\log q^{-1}_0(D))\, \longrightarrow\,\widehat{E}_H(V)\times \text{End}(V)
$$
over $\widehat{E}_H$.
\end{itemize}
(See \cite[Lemma 3.3]{BDM}.)

\begin{definition}\label{def4}
A logarithmic Cartan geometry $(E'_H,\, \Theta')$ is called \textit{weakly homogeneous}
if for every $\rho\, \in\, \mathbb A$ there is a holomorphic isomorphism of
principal $H$--bundles
$$
f_\rho\, :\, E'_H\, \longrightarrow\, \rho^*E'_H
$$
such that
\begin{itemize}
\item $f_\rho$ takes $\Theta'$ to $\rho^*\Theta'$, and

\item the isomorphism $E^V_H \, \longrightarrow\, \rho^* E^V_H$ induced by $f_\rho$
extends to a holomorphic isomorphism
$$
\widehat{E}^V_H \, \longrightarrow\, f^*_\rho \widehat{E}^V_H.
$$
\end{itemize}
\end{definition}

Note that if the isomorphism $E^V_H \, \longrightarrow\, \rho^* E^V_H$ induced by $f_\rho$
extends to a holomorphic homomorphism
$$
\beta\, :\, \widehat{E}^V_H \, \longrightarrow\, \rho^* \widehat{E}^V_H,
$$
then $\beta$ must be an isomorphism. Indeed, the corresponding homomorphism 
$$
\det\beta\, :\, \det \widehat{E}^V_H \,:=\, \bigwedge\nolimits^{\rm top} \widehat{E}^V_H
\, \longrightarrow\, \det \rho^* \widehat{E}^V_H
$$
is given by a holomorphic function on $M$ which is nonzero on $M'$. Therefore,
$\det\beta$ is an isomorphism, which in turn implies that $\beta$ is an isomorphism.

Let $\widetilde{\Gamma}_l$ denote the space of all pairs $(\rho,\, f_\rho)$, where
$\rho\, \in\, {\mathbb A}$ and 
$$
f_\rho\, :\, E'_H\, \longrightarrow\, \rho^*E'_H
$$
is a holomorphic isomorphism of principal $H$--bundles such that
\begin{itemize}
\item $f_\rho$ takes $\Theta'$ to $\rho^*\Theta'$, and

\item the isomorphism $E^V_H \, \longrightarrow\, \rho^* E^V_H$ induced by $f_\rho$
extends to a holomorphic isomorphism
$$
\widehat{E}^V_H \, \longrightarrow\, \rho^* \widehat{E}^V_H.
$$
\end{itemize}
This $\widetilde{\Gamma}_l$ is a complex Lie group. Let
\begin{equation}\label{g2}
\Gamma_l\, \subset\, \widetilde{\Gamma}_l
\end{equation}
denote the connected component containing the identity element.

Let $\text{Aut}_{\Theta'}(E'_H)$ denote the space of all holomorphic automorphisms
of the principal $H$--bundle $E'_H$
$$
f\, :\, E'_H\, \longrightarrow\, E'_H
$$
such that
\begin{itemize}
\item $f$ takes $\Theta'$ to $\Theta'$, and

\item the automorphism $E^V_H \, \longrightarrow\, E^V_H$ induced by $f$
extends to a holomorphic isomorphism
$$
\widehat{E}^V_H \, \longrightarrow\, \widehat{E}^V_H.
$$
\end{itemize}
So $\text{Aut}_{\Theta'}(E'_H)$ is a complex Lie subgroup of $\widetilde{\Gamma}_l$.
Now define
$$
\text{Aut}^0_{\Theta'}(E'_H)\,\,:=\,\,\Gamma_l\cap \text{Aut}_{\Theta'}(E'_H),
$$
where $\Gamma_l$ is defined in \eqref{g2}. We have a short exact sequence of complex Lie groups
\begin{equation}\label{g3}
1\, \longrightarrow\, \text{Aut}^0_{\Theta'}(E'_H)\, \longrightarrow\, \Gamma_l\,
\stackrel{\xi}{\longrightarrow}\, {\mathbb A}\,\longrightarrow\, 1.
\end{equation}

Take any $H$--invariant holomorphic vector field
$$
\gamma\, \in\, H^0(E'_H,\, TE'_H)^H\, \subset\, H^0(E'_H,\, TE'_H)
$$
on $E'_H$. In other words, $\gamma$ defines a holomorphic section of
$\text{At}(E'_H)$ over $M'$. We have
$$
\text{At}(E^V_H)\,=\, (\text{At}(E'_H)\oplus \text{End}(E^V_H))/\text{ad}(E'_H);
$$
see \eqref{x3a} for the inclusion map $\text{ad}(E'_H)\, \hookrightarrow\,
\text{At}(E'_H)$, while the other inclusion map $\text{ad}(E'_H)\, \hookrightarrow\,
\text{End}(E^V_H)$ is obtained from the fact that $E^V_H$ is associated to the
principal $H$--bundle $E_H$ for the $H$--module $V$ for which $H\bigcap \text{kernel}(\chi)$
is a finite group (see \eqref{chi}). Therefore, the above section $\gamma$ of $\text{At}(E'_H)$
over $M'$ and the zero section of $\text{End}(E^V_H)$ together produce a holomorphic section of
$\text{At}(E^V_H)$ over $M'$; this section of $\text{At}(E^V_H)$ will be denoted by
$\widehat{\gamma}$.

The proof of the following lemma is very similar to the proof of Proposition \ref{prop1}.

\begin{lemma}\label{lem3}
The Lie algebra of $\Gamma_l$ consists of all $\gamma\, \in\, H^0(E'_H,\, TE'_H)$ satisfying the
following three conditions:
\begin{enumerate}
\item $\gamma\, \in\, H^0(E'_H,\, TE'_H)^H\, \subset\, H^0(E'_H,\, TE'_H)$,

\item $L_\gamma \Theta \,=\, 0$, and

\item the section $\widehat{\gamma}\, \in\, H^0(M',\, {\rm At}(E^V_H))$ constructed above from $\gamma$
extends to a section of $H^0(M,\, {\rm At}(\widehat{E}^V_H)(-\log D))$.
\end{enumerate}
The Lie bracket operation of ${\rm Lie}(\Gamma_l)$ is given by the Lie bracket operation of vector fields.
\end{lemma}

\begin{proof}
Since any element of the Lie algebra of $\Gamma_l$ gives an holomorphic section of $\text{At}(E'_H)$ over $M'$,
it follows that $\gamma\, \in\, H^0(E'_H,\, TE'_H)^H$. The second condition $L_\gamma \Theta \,=\, 0$ corresponds to the
condition in Definition \ref{def4} that $f_\rho$ takes $\Theta'$ to $\rho^*\Theta'$. The third condition ensures that
the homomorphism $q_*\, :\, H^0(M',\, \text{At}(E'_H))\, \longrightarrow\, H^0(M',\, TM')$, given by the natural projection
\begin{equation}\label{j1}
q\, :\, \text{At}(E'_H)\, \longrightarrow\, TM'
\end{equation}
(see \eqref{x3a}), takes $\gamma$ to $\text{Lie}({\mathbb A})\, \subset\, H^0(M',\, TM')$.
\end{proof}

Let $\Gamma_V$ denote the space of all pairs $(\rho,\, F_\rho)$, where
$\rho\, \in\, {\mathbb A}$ and
$$
F_\rho\, :\, \widehat{E}^V_H\, \longrightarrow\, \rho^*\widehat{E}^V_H
$$
is a holomorphic isomorphism of vector bundles such that $F_\rho$ takes the connection
${\mathcal D}'$ on $E^V_H$ (which is a logarithmic connection on $\widehat{E}^V_H$ (see
Definition \ref{def3})) to the logarithmic connection $\rho^* {\mathcal D}'$ on $\rho^*
\widehat{E}^V_H$. This $\Gamma_V$ is evidently a connected complex Lie group. Recall that the
holomorphic connection ${\mathcal D}'$ on the vector bundle $E^V_H$ on $M'$ is induced by
$\Theta'$. From this it follows immediately that $\widetilde{\Gamma}_l$ in \eqref{g2} is a
complex Lie subgroup of $\Gamma_V$. Indeed, consider the injective  group homomorphism which associate, to each pair $(\rho,\, f_\rho)\, \in\, \widetilde{\Gamma}_l$,  the pair $(\rho,\, F_\rho)$, where $F_\rho$ is the (unique)  extension  of $ f_\rho$ to a bundle
homomorphism $\widehat{E}^V_H \, \longrightarrow\, \rho^* \widehat{E}^V_H$.   The property  that $(\rho,\, F_\rho) \, \in\, \Gamma_V$ comes from the fact that $f_\rho$  sends  $\Theta'$ on  $\rho^*\Theta'$  and therefore sends ${\mathcal D}'$ 
 (which is canonically 
determined by $\Theta'$) on $\rho^* {\mathcal D}'$.

Let $\text{Aut}_V({\mathcal D}')$ denote the group of all
holomorphic automorphisms of $\widehat{E}^V_H$ that take ${\mathcal D}'$ to itself. Now using
\eqref{g3} we have the commutative diagram
\begin{equation}\label{g4}
	\begin{tikzcd}
		1\arrow[r]  & \text{Aut}^0_{\Theta'}(E'_H)\arrow[r] \arrow[d] &  \Gamma_l \arrow[r, "P"]\arrow[d]& {\mathbb A}\arrow[r] \arrow[d, "{\rm Id}"] & 1\\
		1 \arrow[r] & \text{Aut}_V({\mathcal D}')\arrow[r] & \Gamma_V \arrow[r,"\widetilde{P}"]& {\mathbb A }\arrow[r]& 1
	\end{tikzcd}
\end{equation}
where $P$ sends any $(\rho,\, f_\rho)\, \in\, \Gamma_l$ to $\rho$ and
$\widetilde P$ sends any $(\rho,\, F_\rho)\, \in\, \Gamma_V$ to $\rho$. The surjectivity of $P$ and $P'$ in the above diagram follow from the assumption that  $(E_H',\Theta')$ is weakly homogeneous.
note that all the vertical arrows in \eqref{g4} are injective. 

\begin{theorem}\label{thm1}
Let $(E'_H,\, \Theta')$  be a  weakly homogeneous logarithmic Cartan geometry. The holomorphic principal
$H$--bundle $E'_H$ admits a holomorphic connection
$\nabla$ such that $\Theta'$ is preserved by $\nabla$. Moreover, there is a logarithmic
connection on $\widehat{E}^V_H$ which induces such a holomorphic connection $\nabla$ on $E'_H$.
\end{theorem}

\begin{proof}
Let $dP\, :\, \text{Lie}(\Gamma_l)\, \longrightarrow\, \text{Lie}({\mathbb A})\,=\, H^0(M,\, TM(-\log D)$ be the
homomorphism of Lie algebras corresponding to the projection $P$ in \eqref{g4}. The Lie algebra $\text{Lie}(\Gamma_l)$ is
described in Lemma \ref{lem3}. Fix a homomorphism of complex vector spaces 
\begin{equation}\label{g5}
\Phi\,:\, H^0(M,\, TM(-\log D)) \, \longrightarrow\, \text{Lie}(\Gamma_l)
\end{equation}
such that $dP\circ \Phi\,=\, \text{Id}_{H^0(M, TM(-\log D))}$; note that $dP$ is surjective because $P$ is so.

The homomorphism $\Phi$ in \eqref{g5} produces a holomorphic homomorphism
\begin{equation}\label{g6a}
\widehat{\Phi}\,:\, TM' \, \longrightarrow\, \text{At}(E'_H)
\end{equation}
over $M'$ which will be described below. Take any $x\,\in\, M'$ and $w\, \in\, T_xM'$. Since $TM(-\log D)$ is
holomorphically trivial, there is a unique section
$$
\widetilde{w}\, \in\, H^0(M,\, TM(-\log D))
$$
such that $\widetilde{w}(x)\,=\, w$. Now define $\widehat{\Phi}$ in \eqref{g6a} as follows:
$$
\widehat{\Phi}(w)\,:=\, \Phi(\widetilde{w})(x)\, \in\, \text{At}(E'_H)_x.
$$
In order to see that the image $\text{Lie}(\Gamma_l)$  of  $\Phi(\widetilde{w})$ lies in $ \text{At}(E'_H)$, recall  that  statement (1) in     Lemma \ref{lem3}   shows that any element   $\gamma\, \in \text{Lie}(\Gamma_l)$  is  a $H$--invariant holomorphic vector field
$$
\gamma\, \in\, H^0(E'_H,\, TE'_H)^H\, \subset\, H^0(E'_H,\, TE'_H)
$$
on $E'_H$. In other words, $\gamma$ defines a holomorphic section of
$\text{At}(E'_H)$ over $M'$.

Since $\Phi$ in \eqref{g5} satisfies the condition $dP\circ \Phi\,=\, \text{Id}_{H^0(M, TM(-\log D))}$, it follows
immediately that
$$
q\circ \widehat{\Phi}\,=\, {\rm Id}_{TM'},
$$
where $q$ is the projection in \eqref{j1}. Therefore, $\widehat{\Phi}$ defines a holomorphic connection
on the principal $H$--bundle $E'_H$. This connection on $E'_H$ given by $\widehat{\Phi}$ will be denoted by $\nabla$.

Since any $\gamma\, \in\, \text{Lie}(\Gamma_l)$ satisfies the condition that
\item $L_\gamma \Theta \,=\, 0$ (see Lemma \ref{lem3}), it follows that $\nabla$ preserves $\Theta'$.

The diagram in \eqref{g4} gives the following commutative diagram of homomorphism of Lie algebras:
\begin{equation}\label{g6}
\begin{matrix}
\text{Lie}(\Gamma_l) & \stackrel{dP}{\longrightarrow} & H^0(M,\, TM(-\log D)) & \longrightarrow & 0\\
\,\,\,\Big\downarrow J && \,\,\,\,\Big\downarrow {\rm Id}\\
\text{Lie}(\Gamma_V) & \stackrel{d\widetilde P}{\longrightarrow} & H^0(M,\, TM(-\log D)) &\longrightarrow & 0
\end{matrix}
\end{equation}
where $J$ is the homomorphism of Lie algebras corresponding to the injective homomorphism
$\Gamma_l\, \longrightarrow\, \Gamma_V$ in \eqref{g4}.

Consider the homomorphism
$$
J \circ\Phi\,:\, H^0(M,\, TM(-\log D)) \, \longrightarrow\, \text{Lie}(\Gamma_V),
$$
where $\Phi$ and $J$ are the homomorphisms in \eqref{g5} and \eqref{g6} respectively. It produces a holomorphic
homomorphism
\begin{equation}\label{g7}
\widetilde{J}\, :\, TM(-\log D) \, \longrightarrow\, \text{At}(\widehat{E}^V_H)(-\log D),
\end{equation}
which will now be described. Take any $x\,\in\, M$ and $w\, \in\, TM(-\log D)_x$. Since $TM(-\log D)$ is holomorphically trivial, there
is a unique section $\widetilde{w}\, \in\, H^0(M,\, TM(-\log D))$
such that $\widetilde{w}(x)\,=\, w$. Define $\widetilde{J}$ in \eqref{g7} as follows:
$$
\widetilde{J}(w)\,:=\, J \circ\Phi(\widetilde{w})(x)\, \in\, \text{At}(\widehat{E}^V_H)(-\log D)_x.
$$

In order to see that the image $\text{Lie}(\Gamma_l)$  of  $J \circ\Phi(\widetilde{w})$  lies in $\text{At}(\widehat{E}^V_H)(-\log D)$, recall  that  statement (3) in     Lemma \ref{lem3}   shows that any element   $\gamma\, \in \text{Lie}(\Gamma_l)$ 
 canonically defines a holomorphic section of $H^0(M,\, {\rm At}(\widehat{E}^V_H)(-\log D))$.

Since $\Phi$ in \eqref{g5} satisfies the condition $dP\circ \Phi\,=\, \text{Id}_{H^0(M, TM(-\log D))}$, from \eqref{g6}
it follows immediately that
$$
q'\circ \widetilde{J}\,=\, {\rm Id}_{TM (-\log D)},
$$
where $q'\, :\, \text{At}(\widehat{E}^V_H)(-\log D)\, \longrightarrow\, TM(-\log D)$ is the natural projection
(see \eqref{x14}). Therefore, $\widetilde{J}$ defines a logarithmic connection on $\widehat{E}^V_H$ singular over $D$.
The restriction of this logarithmic connection to $E^V_H\, \longrightarrow\, M'$ clearly coincides with the holomorphic
connection on $E^V_H$ induced by the connection $\nabla$ on the principal $H$--bundle $E'_H$ defined
by $\widehat{\Phi}$ in \eqref{g6a}. Therefore, the logarithmic connection on $\widehat{E}^V_H$ defined by $\widetilde{J}$
induces the connection on the principal $H$--bundle $E'_H$ defined by $\widehat{\Phi}$.
This completes the proof.
\end{proof}

\begin{example}\label{ex2}
As in Example \ref{ex1}, consider the standard action of $\mathbb{C}^*\,=\,
{\mathbb C}\setminus\{0\}$ on $\mathbb{C}\mathbb{P}^1$. Set $M\,=\, \mathbb{C}\mathbb{P}^1$, and
$D\,=\, \{0,\, \infty\}$. There is no affine
structure on $\mathbb{C}\mathbb{P}^1$. But there is a $1$-parameter family of weakly homogeneous
logarithmic affine structures on $\mathbb{C}\mathbb{P}^1$. Note that while $H^0(X,\,
K_{\mathbb{C}\mathbb{P}^1})\,=\, 0$, we have $\dim H^0(X,\,
K_{\mathbb{C}\mathbb{P}^1}\otimes {\mathcal O}_{\mathbb{C}\mathbb{P}^1}(0+\infty))\,=\, 1$, and
each holomorphic section of $K_{\mathbb{C}\mathbb{P}^1}\otimes {\mathcal O}_{\mathbb{C}\mathbb{P}^1}
(0+\infty)$ is fixed by the above action of $\mathbb{C}^*$ on $\mathbb{C}\mathbb{P}^1$.
\end{example}

\section*{Acknowledgements}

We thank the referee for helpful comments to improve the exposition.
The first-named author is partially supported by a J. C. Bose Fellowship (JBR/2023/000003).
The third-named author would like to thank the Institute of Eminence, University of Hyderabad (UoH-IoE-RC5-22-003) for the partial support in the form of a grant.


\end{document}